\documentclass[11pt,a4paper,headinclude,footinclude,fleqn,reqno]{amsart}                 
\usepackage[T1]{fontenc}                   
\usepackage[utf8]{inputenc}                 
\usepackage[english]{babel}       
\usepackage{graphicx}                      %
\usepackage[font=small]{quoting}            %
\usepackage{caption}  
     \usepackage{amsmath}
          \usepackage{amssymb}
\usepackage{textcomp}

 \usepackage[top=1in,bottom=1.5in,left=1.2in,right=1.2in]{geometry}

\usepackage{verbatim}
\usepackage{color}
\usepackage{picture}
\usepackage{graphicx}
\usepackage{graphics}
\usepackage{comment}
\usepackage{hyperref}
\hypersetup{colorlinks,linkcolor={blue},citecolor={blue},urlcolor={red}}

\newtheorem{thm}{Theorem}[section] 
 
\newtheorem{cor}[thm]{Corollary} 
\newtheorem{prop}[thm]{Proposition}

\theoremstyle{definition} 
 
\theoremstyle{remark}

\usepackage[bottom]{footmisc}

\def\S{\Sigma} 
\def\n{\nabla}

\def\p{\partial}

\def\n{\nabla}

\def\C{\mathcal{C}}

\def\p{\partial}

\def\s{\sigma}

\def\n{\nabla}
\def\<{\langle}
\def\>{\rangle}

\def\n{\nabla}

\def\NN{\mathbb{N}}

\def\RR{\mathbb{R}}

\def\SS{\mathbb{S}}

\def\L{\mathcal{L}}
\def\p{\partial}

\def\s{\sigma}

\def\V{\mathcal V}

\def\T{\mathcal T}

\def\bt{\mathbf{b}_\theta}

\def\R{\mathbb{R}}

\DeclareMathOperator{\dive}{div}

{\left\{\begin{array}{@{}l@{}}}{\end{array}\right.}
\patchcmd{\abstract}{\scshape\abstractname}{\textbf{\abstractname}}{}{}
\makeatletter 
\def\@makefnmark{} 
\makeatother

\numberwithin{equation}{section}

\makeatletter
\@namedef{subjclassname@2020}{\textup{2020} Mathematics Subject Classification}
\makeatother

\begin{document}
\title[A Minkowski-type inequality]{ A Minkowski-type inequality for\\ capillary hypersurfaces in a half-space}
 \author[G. Wang]{Guofang Wang}
 
\address[G. Wang]{Mathematisches Institut, Albert-Ludwigs-Universit\"{a}t Freiburg, Freiburg im Breisgau, 79104, Germany}
\email{guofang.wang@math.uni-freiburg.de}

 \author[L. Weng]{Liangjun Weng}
 
\address[L. Weng]{School of Mathematical Sciences, Anhui University, Hefei, 230601, P. R. China}

\address{ Dipartimento Di Matematica, Universit\`a degli Studi di Roma "Tor Vergata", Via della Ricerca Scientifica 1, 00133,  Roma, Italy}
\email{ljweng08@mail.ustc.edu.cn}

 \author[C. Xia]{Chao Xia}
\address[C. Xia]{School of Mathematical Sciences, Xiamen University, Xiamen, 361005, P. R. China}
\email{chaoxia@xmu.edu.cn}

 
\maketitle
\begin{abstract}
In this article, we investigate a flow of inverse mean curvature type for capillary hypersurfaces in the half-space. We establish the global existence of solutions for this flow and demonstrate that it converges smoothly to a spherical cap as time tends to infinity. As a result, we derive a new Minkowski-type inequality for star-shaped and mean convex capillary hypersurfaces for the whole range of contact angle $\theta\in(0,\pi)$.
\end{abstract}


\section{Introduction}
Let $\RR^{n+1}_+=\{x\in \RR^{n+1} \, |\, x_{n+1} > 0\}$ $(n\ge 2)$ be the  Euclidean upper half-space and $\Sigma$  an embedded compact hypersurface in  $\overline{\RR^{n+1}_+}$ with boundary $\p\Sigma$ lying in $\partial \RR^{n+1}_+$. Such a hypersurface is called a
\textit{capillary hypersurface}, if 
$\S$ intersects with $\p {{\R}}_+^{n+1}$ at a constant contact angle $\theta\in (0,\pi)$ along $\p\S$.  In this paper, we consider a family of embedded capillary hypersurface $\S_t:=x(M,t)\subset \overline{\RR^{n+1}_+}$, where $M$ is an $n$-dimensional compact orientable smooth manifold with boundary $\p M$, satisfies the following inverse mean curvature type flow
\begin{eqnarray}\label{flow00}
\left\{
\begin{array}{llll}
(\partial_t x)^\perp  &=& \left( \frac {n\left(1+\cos \theta \<\nu, e\> \right)}{H} -\<x, \nu\>\right) \nu, \quad &
\hbox{ in } M\times[0,T),\\
\<\nu ,e \> &=& -\cos\theta 
\quad  & \hbox{ on }\partial M \times [0,T),
\end{array} \right.
\end{eqnarray}where $\nu$ is the unit outward normal of 
$\S_t$ and $e:=(0,\cdots, 0, -1)$.

 The introduction of flow \eqref{flow00} is motivated by the idea of Guan-Li \cite{GL09, GL15} (see also \cite{BGL, CGLS, GLW}) for establishing the isoperimetric-type inequality for quermassintegral and by the Minkowski formulas for capillary hypersurfaces in \cite{WWX2022}. For more description of Guan-Li's idea for free boundary or capillary hypersurfaces, we refer to \cite{MWW, Sch2,  SWX, WWX2022, WeX21} and references therein.

For a capillary hypersurface $\S\subset \overline{\RR^{n+1}_+}$, we denote $\widehat \Sigma$  the bounded domain in $\overline{\RR^{n+1}_+}$ enclosed by $\S\subset \overline{\RR^{n+1}_+}$ and $\p \RR^{n+1}_+$. The boundary of $\widehat{\S}$  consists of two parts: one is $\Sigma$ and the other, which
will be denoted by $\widehat{\p \Sigma}$,  lies on   $\p \RR^{n+1}_+$.
Both have a common boundary, namely $\p \Sigma$.  The capillary area functional (cf. \cite{Finn}) is defined by  \begin{eqnarray*} 	
\V_{1,\theta}(\widehat{\Sigma}):=  |\S|-\cos\theta |\widehat{\p\S}|.\end{eqnarray*} 
For a hypersurface with boundary (not necessarily capillary) in $\overline{\RR^{n+1}_+}$, the relative (or capillary) isoperimetric inequality says that\begin{eqnarray*} 	
\frac{\V_{1,\theta}(\widehat{\Sigma})}{|\widehat{\S}|^{\frac{n}{n+1}}}\ge \frac{\V_{1,\theta}(\widehat{\C_{\theta, 1}})}{|\widehat{\C_{\theta, 1}}|^{\frac{n}{n+1}}}=(n+1) \bt^{\frac{1}{n+1}},\end{eqnarray*}
with equality holding if and only if $\widehat{\Sigma}$ is a capillary spherical cap.
Here we use $\C_{\theta, r}$ to denote a capillary spherical cap lying entirely in $\overline{\RR^{n+1}_+}$ by
\begin{eqnarray}\label{sph-cap}
\C_{\theta, r}:=\Big\{x\in \overline{\RR^{n+1}_+}\big||x-r\cos\theta e|=r \Big\},~r~\in (0, \infty),
\end{eqnarray} 
which is a portion of the sphere of radius $r$ and centered at $r\cos\theta e$. For notation simplicity, we denote  $$(n+1) \bt:=\V_{1,\theta}(\widehat{\C_{\theta,1}}).$$ One can check easily that $\bt$ is the volume of $\widehat{\C_{\theta, 1}}.$

In \cite{WWX2022}, we introduced a family of quermassintegrals for capillary hypersurfaces in the half-space, which can be regarded as a natural generalization of quermassintegrals for closed hypersurfaces. Among them the second quermassintegral is  given by \begin{eqnarray}\label{quermassintegrals}	\V_{2,\theta}(\widehat{\S}):=	\frac{1}{n}\left( \int_\S H dA -  \cos\theta \sin\theta    |\p\S| \right)	.\end{eqnarray} 
The main objective of this paper is to establish an optimal isoperimetric type inequality between $\V_{2,\theta}(\widehat{\S})$ and the capillary area functional $\V_{1,\theta}(\widehat{\S})$ under star-shaped and mean convex assumption, which is a capillary counterpart of the classical Minkowski-type inequality for closed hypersurfaces. For the latter, we refer to  Schneider's seminal book on convex bodies \cite{Schneider} and Guan-Li \cite{GL09}.

We return to flow \eqref{flow00}. A key feature of flow \eqref{flow00} is that it preserves the capillary area functional $\V_{1,\theta}(\widehat{\S_t})$ while it monotonically decreases $\V_{2,\theta}(\widehat{ \S_t})$ (cf. \cite[Proposition 3.1]{WWX2022}). It is easy to check that a family of stationary solutions of flow \eqref{flow00} is given by capillary spherical caps $\C_{\theta, r}$.
Our main result in this article is the following long-time existence and convergence for flow \eqref{flow00} under the assumption of star-shapedness and mean convexity. This can be viewed as a capillary counterpart of the results of Gerhardt \cite{Ger} and Urbas \cite{Urb90}  for closed hypersurfaces.
\begin{thm}\label{k-convex conv} 
	Assume the initial hypersurface  $\S_0$ is a star-shaped and strictly mean convex capillary hypersurface in $\overline{\RR^{n+1}_+}$ $(n\ge 2)$ with a contact angle   $\theta \in (0, {{\pi}} )$. 	Then  flow \eqref{flow00} starting from $\S_0$ exists for all time with uniform $C^{\infty}$-estimates. Moreover, it converges smoothly to a uniquely determined capillary spherical cap $\C_{\theta, r}$, as $t\to+\infty$.	
\end{thm}
 As an application, we obtain the following Minkowski-type inequality for star-shaped and mean convex capillary hypersurfaces in $\overline{\RR^{n+1}_+}$ for the whole angle range $ {\theta \in (0, \pi)}$.
\begin{thm}\label{thm 2}
	Let $\Sigma\subset \overline{\RR^{n+1}_+}$ $(n\ge 2)$ be a  star-shaped and mean convex capillary hypersurface with a contact angle $ {\theta \in (0, \pi)}$, then 
	\begin{eqnarray}\label{mink ineq}
	\int_\S HdA \geq n(n+1)^{\frac {1}{n}} \bt^{\frac{1}{n}} (|\S|-\cos\theta |\widehat{\p\S}|)^{\frac{n-1}{n}}+\sin\theta\cos\theta |\p\S|.
	\end{eqnarray} 
 Equality holds if and only if $\Sigma$  is a capillary spherical cap. 
 
\end{thm} 
Note that \eqref{mink ineq} is equivalent to 
\begin{eqnarray}\label{mink ineq1}	\frac{\V_{2,\theta}(\widehat{\S})}{\V_{1,\theta}(\widehat{\S})^{\frac{n-1}{n}}} \geq \frac{\V_{2,\theta}(\widehat{\C_{\theta, 1}})}{\V_{1,\theta}(\widehat{\C_{\theta, 1}})^{\frac{n-1}{n}}} =n(n+1)^{\frac {1}{n}} \bt^{\frac{1}{n}}.	\end{eqnarray}
From \eqref{mink ineq1} one can see easily that the functional $\frac{\V_{2,\theta}(\widehat{\S})}{\V_{1,\theta} (\widehat{\S})^{\frac {n-1} n}} $ achieves its minimum at spherical caps. 

When  $\theta= \pi \slash 2$, Theorem \ref{thm 2} follows from the result of  Guan-Li \cite[Theorem 2]{GL09} for closed hypersurfaces by using a 
reflection argument. It is clear that if $\theta\not = \pi\slash 2$ the reflection argument does not work.

Under the stronger geometric condition that $\Sigma$ is convex, Theorem \ref{thm 2}   has been proved for contact angle $\theta\in (0, \pi\slash  2]$ in \cite{Wei_el}, where the authors proved actually more general geometric inequalities, the Alexandrov-Fenchel inequalities, which were initiated and conjectured in \cite{WWX2022}. We emphasize that the results in \cite{WWX2022} and \cite{Wei_el} crucially depend on the restriction $\theta \le \pi/2$, which was used to obtain $C^2$ estimates.
We observed in \cite{MWWX} that for convex capillary hypersurfaces, all the Alexandrov-Fenchel inequalities with the whole range of the contact angle $\theta \in (0, \pi)$ follow actually from the classical Alexandrov-Fenchel inequalities for non-smooth convex bodies and provided a proof in a smooth setting with a Robin-type boundary condition. The proof follows closely the original idea {in the theory of convex geometry} with necessary modifications on the boundary.

It remains to be asked if the convexity condition could be weakened for these Alexandrov-Fenchel inequalities for capillary hypersurfaces.  This paper is the first one in this direction.
We remark that without the star-shaped assumption, the Minkowski inequality for closed mean convex hypersurfaces remains an open problem, despite numerous efforts made to solve it, see \cite{AFM, Glau, Hui} for instance. 
 It is also natural to expect that \eqref{mink ineq} holds for general mean convex capillary hypersurfaces. Especially a capacity method used in \cite{AFM} would be expected to be useful by a necessary modification.  We also mention several works on the Minkowski-type inequality in warped product spaces  \cite{BHW, Wei, Sch2}, as well as anisotropic Minkowski-type inequality \cite{Xia}.

In order to prove Theorem \ref{k-convex conv}, it is crucial to show that flow \eqref{flow00} preserves the star-shapedness and strictly mean convex. For the star-shapedness, we prove a uniform lower bound for a new test function 
$$\bar u:=\frac{\<x,\nu\>}{1+\cos\theta \<\nu,e\>},$$ which is a variant of the support function $\<x,\nu\>$. In fact, it can be viewed as the anisotropic support function (see e.g. \cite{Xia}) with respect to ${\mathcal C}_{\theta, 1}$.
Another key point is the uniform lower and upper bounds for $H$. For this, instead of $H$ we consider  $$P=\bar u H.$$ To get the curvature estimate,  we employ a similar method in \cite{Marquardt2013} and \cite{Xia} with a few modifications to take care of the Robin boundary condition, which explore the specific divergence structure of mean curvature, to directly obtain Schauder's estimate for the radial function for star-shaped hypersurfaces.  This completes the proof of Theorem \ref{k-convex conv}. For the proof of Theorem \ref{thm 2}, it follows from the monotonicity property $\V_{2,\theta}$ along flow \eqref{flow00} and the convergence result in Theorem \ref{k-convex conv}.

The paper is organized as follows: In Section \ref{sec2}, some relevant evolution equations are collected and we prove the preservation of the star-shapedness and obtain uniform a priori estimates for flow \eqref{flow00}. Section \ref{sec4} is dedicated to demonstrating the long-time existence and global convergence of flow \eqref{flow00}, specifically, we prove Theorem \ref{k-convex conv}. As a result, we conclude the proof of Theorem \ref{thm 2}.

\

\section{Mean convexity and star-shapedness}\label{sec2}

\subsection{Evolution equations}

Let $\Sigma_t$ be a family of smooth, embedded hypersurfaces with capillary boundary in $ \overline{\RR^{n+1}_+}$, given by the embeddings $x(\cdot,t): M\to \overline{\RR^{n+1}_+}$, which evolve by the general flow.

\begin{eqnarray}\label{flow with normal and tangential}
\p_t x=f\nu+\T,
\end{eqnarray} with $\T\in T\Sigma_{t}$ and a general velocity function $f$. For a fixed $f$, we choose $
\T |_{\p \S_t}=f\cot\theta \mu$ with $\mu$ being the unit outward conormal of $\p\S_t$ in $\S_t$, which implies that the restriction of $x(\cdot, t)$ on $\p M$ is contained in $\RR^n$. We refer to \cite[Section 2]{WWX2022} for further discussions. Along flow \eqref{flow with normal and tangential}, we have the following evolution equations for the induced metric $g_{ij}$, the volume element $d\mu_t$,  the unit outward normal $\nu$, the second fundamental form $(h_{ij})$, the Weingarten curvature tensor $(h^i_j)$ and the mean curvature $H$ of the hypersurfaces $\Sigma_t:=x(M, t)\subset  \overline{\RR^{n+1}_+}$.   (cf. \cite[Proposition 2.11]{WeX21} for a proof)
\begin{prop}[\cite{WeX21}]\label{basic evolution eqs}
	Along  flow \eqref{flow with normal and tangential}, there holds  
	\begin{enumerate} 
		\item $\p_t g_{ij}=2fh_{ij}+\n_i \T_j+\n_j\T_i$.
  	\item $\p_t d\mu_t=(fH+\dive (\T)) d\mu_t$.
   
\item $\p_t\nu =-\n f+h(e_i,\T)e_i$.	
		\item $\p_t h_{ij}=-\n^2_{ij}f +fh_{ik}h_{j}^k +\n_\T h_{ij}+h_{j}^k\n_i\T_k+h_{i}^k\n_j \T_k.$
		\item $\p_t h^i_j=-\n^i\n_{j}f -fh_{j}^kh^{i}_k+\n_\T h^i_j.$
		\item $\p_t H=-\Delta f-|h|^2 f+ \langle\n  H, \T\rangle $.

	\end{enumerate}In above, $\n$ and $\Delta$ are the Levi-Civita connection and Beltrami-Laplacian operator on $\S_t$ with respect to the induced metric.
\end{prop}  

\subsection{Star-shapedness and Mean convexity}

We turn to the study of the flow.
In this subsection, we show that the star-shapedness and the mean convexity are preserved along flow \eqref{flow00}. In order to prove this, the key is a choice of a suitable test function. 

In the following  we study flow \eqref{flow00}, i.e, namely we consider 
$$f=\frac {n\left(1+\cos \theta \<\nu, e\> \right)}{H} -\<x, \nu\>.$$
For convenience, we introduce the parabolic operator with respect to  flow \eqref{flow00} as
\begin{eqnarray*}
	\L :=\p_t-\frac{n(1+\cos\theta\langle \nu,e\rangle)}{H^2}\Delta-\left\langle \T+x-\frac{n\cos\theta}{H}e,\n \right\rangle. 
\end{eqnarray*}

Let $\Sigma_0$ be an initial capillary hypersurface which is star-shaped and mean convex.
The star-shapedness of $\S_0$ implies that there exist  $0<r_1<r_2<\infty$ such that 
$$\S_0\subset \widehat{\C_{\theta,r_2}}\setminus \widehat{\C_{\theta,r_1}},$$ 	where $\widehat{\C_{\theta,r}}$ denotes the  bounded domain in $ \overline{\RR^{n+1}_+}$ enclosed by $\C_{\theta,r}$ and $\p\R^{n+1}_+$.
Following the argument in \cite[Proposition 4.2]{WWX2022}, we have the following height estimate, which follows from the maximum principle (or the avoidable principle) since $\C_{\theta, r}$ is a stationary solution of flow \eqref{flow00}.
\begin{prop}\label{c0 est}
	For any $t\in [0, T)$, along  flow \eqref{flow00},  there holds
	\begin{eqnarray}\label{c0 bound}
		\Sigma_{t}\subset \widehat{\C_{\theta,r_2}}\setminus\ \widehat{\C_{\theta,r_1}},
	\end{eqnarray}
	where $\C_{\theta,r}$ defined by \eqref{sph-cap} and $r_{1}, r_{2}$ only depend on $\Sigma_{0}$.
\end{prop} \label{prop_3.2}
Next, we show that the star-shapedness is preserved along flow \eqref{flow00}.
\begin{prop}	\label{c1-est to flow}
	Let $\S_0$ be a  star-shaped and strictly mean convex hypersurface with capillary boundary in $\overline{\RR^{n+1}_+}$ and $\theta \in(0,\pi)$, then there exists $c_0>0$ depending only on $\S_0$, such that the solution $\S_t$ of flow \eqref{flow00} satisfies
	\begin{eqnarray}\label{c1 est flow}
		\langle x,\nu\rangle(p,t) \geq c_0,
	\end{eqnarray}for all $(p,t)\in M\times [0,T)$. 
\end{prop}

\begin{proof}
A	direct computation  gives
	\begin{eqnarray*}
		\langle x,\nu\rangle_{;ij}  = h_{ij}+h_{ij;k}\langle x,e_k\rangle -(h^2)_{ij}\langle x,\nu\rangle .
	\end{eqnarray*}
	Using Proposition \ref{basic evolution eqs}, we have 
	\begin{eqnarray*}
		\p_t\langle x,\nu\rangle &=&\langle f\nu+\T,\nu\rangle +\langle x,-\n f+h(e_i,\T)e_i\rangle 
	\\&=&\frac{n(1+\cos\theta\langle \nu,e\rangle)}{H}-\langle x,\nu\rangle +h(x^T,\T) \\&&-\langle x,e_k\rangle \left(\frac{n\cos\theta \langle h_{ki}e_i,e\rangle}{H}- (1+\cos\theta \langle \nu,e\rangle)   \frac{nH_{;k}}{H^2}-\langle x,h_{ki}e_i\rangle \right),
	\end{eqnarray*}
	which yields 
	\begin{eqnarray}\label{evo of support fun}
		\L \langle x,\nu\rangle = \Big[ \big(1+\cos\theta \langle \nu,e\rangle\big)\frac{n|h|^2}{H^{2}} -1\Big]\langle x,\nu\rangle.
	\end{eqnarray}
It is easy to see	\begin{eqnarray*}
		\langle \nu,e\rangle_{;kl}= h_{kl;s}\langle e_s,e\rangle -(h^2)_{kl}\langle \nu,e\rangle .
	\end{eqnarray*}
Combining  it with
	\begin{eqnarray*}	\<e,\n f\>&=& n\<e,e_i\>   \left(\frac{1+\cos\theta \<\nu,e\>}{H}\right)_{;i}-\<e,e_i\> \<x,\nu\>_{;i}
		\\&=&\frac{n\cos\theta}{H} h(e^T,e^T)-nH^{-2}(1+\cos\theta \<\nu,e\>) \<\n H,e\>-h(x^T,e^T),	\end{eqnarray*}
	we have
	\begin{eqnarray*}
		\p_t \<\nu,e\>&=& -\<\n f,e\>+h(e_i,\T)\<e_i,e\>
		\\&=&-\frac{n\cos\theta}{H} h(e^T,e^T) +nH^{-2}(1+\cos\theta \<\nu,e\>) \<\n H,e\>+h(x^T,e^T)+h(\T,e^T).
	\end{eqnarray*}
It follows
	\begin{eqnarray}\label{evo of nu e}
		\L \<\nu,e\>= (1+\cos\theta \<\nu,e\>)nH^{-2} |h|^2 \<\nu,e\>.
	\end{eqnarray}
	Now we introduce the function mentioned in the introduction  
	\begin{eqnarray} \label{fun_u}
\bar u:=\frac{\<x,\nu\>}{1+\cos\theta \<\nu,e\>}.
	\end{eqnarray}
	One can check that it  satisfies
	\begin{eqnarray*}
 \L \bar u&=& \frac{1}{1+\cos\theta \<\nu,e\>} \L u -\frac{\<x,\nu\>}{(1+\cos\theta\<\nu,e\>)^2} \L (1+\cos\theta \<\nu,e\>) \notag \\&&+\frac{2n (1+\cos\theta \<\nu,e\>)}{H^2} \Bigg[\frac{\<x,\nu\>_{;i} (1+\cos\theta\<\nu,e\>)_{;i}}{(1+\cos\theta \<\nu,e\>)^2} \\&&- \<x,\nu\> \frac{(\cos\theta\<\nu,e\>)_{;i}(\cos\theta\<\nu,e\>)_{;i}}{(1+\cos\theta \<\nu,e\>)^3}\Bigg].
	\end{eqnarray*}
Combining with equations \eqref{evo of support fun} and \eqref{evo of nu e}, we get
	\begin{eqnarray}\label{evo of baru}
	\L \bar u=\frac{\<x,\nu\>}{1+\cos\theta \<\nu,e\>}\left(\frac{n|h|^2}{H^2}-1\right)	 +\frac{2n }{H^2} {\bar u}_{;i} (1+\cos\theta \<\nu,e\>)_{;i}.
\end{eqnarray}
On $\p M$, we have  \begin{eqnarray*} 
		\n_\mu \langle x,\nu\rangle  =\langle x, h(\mu,\mu)\mu\rangle =\cos \theta h(\mu,\mu)\langle x, \overline \nu\rangle =  \cot\theta h(\mu,\mu)\langle x,  \nu\rangle 
	\end{eqnarray*} 
	and
	\begin{eqnarray*}
		\n_\mu\big(1+\cos\theta \langle \nu,e \rangle \big)= -\sin\theta h(\mu,\mu)\langle \nu,e\rangle.
	\end{eqnarray*}
Altogether  yields
	\begin{eqnarray}\label{neuman of bar u}
		\n_{\mu} \bar u=0, \ \text{ on } \ \p M.
	\end{eqnarray}
Since
\begin{eqnarray}\label{eq_a2}
n|h|^2 \geq H^2,
\end{eqnarray}
\eqref{evo of baru} gives
\begin{eqnarray*}
	\L \bar u\geq 0,\quad \text{ mod } \n \bar u,
\end{eqnarray*}
which, together with 
 \eqref{neuman of bar u}, implies  
	\begin{eqnarray*}
		\bar u\geq \min_{M} \bar u(\cdot, 0)
			\end{eqnarray*} by the maximum principle.
Since $1+\cos\theta \<\nu,e\>\geq 1-\cos\theta >0$, then
	\begin{eqnarray*}
		\<x,\nu\>=(1+\cos\theta \<\nu,e\>) \cdot \bar u\geq c_0>0,
	\end{eqnarray*}for some positive constant $c_0$, which  depends  only on $\S_0$.	
\end{proof}
Next, we show the uniform bound of the mean curvature and the mean convexity is preserved along flow \eqref{flow00}.

The evolution equation of $H$ was computed in \cite[Proposition 4.3]{WWX2022}.
\begin{prop}[\cite{WWX2022}]\label{evol of F bdry} Along  flow \eqref{flow00}, there holds
	\begin{eqnarray}\label{evol of F}
	\L H&=&  2n H^{-2} \cos\theta h(\n H,e^T)-2nH^{-3}(1+\cos\theta \<\nu,e\>) |\n H|^2+H 
 \left(1-\frac{n|h|^2}{H^2}\right),\notag
\end{eqnarray}
and  
\begin{eqnarray}\label{neumann of F}
	\n_\mu H=0,\quad \text{ on } \p M.
\end{eqnarray}
\end{prop} 
\begin{prop}\label{upper bdd of F}
	If $\Sigma_{t}$ solves  flow \eqref{flow00}, 
	then
	\begin{eqnarray}\label{F curv lower bound}
		H(p,t)\leq \max_M H(\cdot, 0).
\end{eqnarray} \end{prop}
\begin{proof}
Using \eqref{eq_a2}, by \eqref{evol of F}, we see
	\begin{eqnarray*}
		\L H =H\left(1-\frac{n|h|^2}{H^2}\right) \leq 0,\quad \text{ mod } \n H.
	\end{eqnarray*}
	Since $\n_\mu H=0$ on $\p M$, the conclusion directly follows from the maximum principle.	
\end{proof}

From the evolution equation for $H$ one could not obtain the lower bound of $H$.  Instead, we consider the function $H\bar u$.
\begin{prop}\label{lower bds of F}
	If $\Sigma_{t}$ solves  flow \eqref{flow00} with an initial hypersurface $\Sigma_{0}$ being a  star-shaped and strictly mean convex capillary hypersurface   in $\overline{\RR^{n+1}_+}$, then
	\begin{eqnarray}\label{mean curv lower bound imcf}
	H(p,t)\geq C,\quad\quad \forall (p, t)\in M\times [0,T),
	\end{eqnarray}	 
	where the positive constant  $C$ depends only on the initial datum.
\end{prop}

\begin{proof}
	Define  the function 
	\begin{eqnarray}\label{fun_P}P:=H \bar u .
	\end{eqnarray}	
	Using  \eqref{evo of baru} and \eqref{evol of F}, we obtain
	\begin{eqnarray*}
		\L P&=& \bar u\mathcal{L}H+H\mathcal{L} \bar u-2 n(1+\cos\theta \<\nu,e\>)\frac{\<\n  H,\n \bar u\>  }{H^2} \\
	\\&=& \frac{2n}{H^{2}}\left\< \n (1+\cos\theta \<\nu,e\>) -(1+\cos\theta \<\nu,e\>)  \frac{\n H}{H}, \n P\right\>.
	\end{eqnarray*}
	On $\p M$,   \eqref{neuman of bar u} and \eqref{neumann of F} imply  \begin{eqnarray}\label{neuman of P imcf}
		\nabla_{\mu}P= 0. 
	\end{eqnarray}
By applying the maximum principle, we get
	\begin{eqnarray*}
		P\geq \min_{M} H \bar u(\cdot,0),
	\end{eqnarray*}together  with the uniform upper bound of  $\bar u$ (which follows from Proposition \ref{prop_3.2}), it implies the desired uniform estimate of $H$.
\end{proof}

\
\

\section{Long-time existence and convergence}\label{sec4}
	The short-time existence of flow \eqref{flow00} is well-known.  Let $T^*$ be the maximal time of the  existence of a smooth  solution to \eqref{flow00}
in the class of star-shaped and strictly mean convex hypersurfaces. In order to establish the long-time existence of flow \eqref{flow00}, we need to obtain the uniform height estimate, the gradient estimate, and the second order derivative estimate (or curvature estimates) for the solution of flow \eqref{flow00}. Then the long-time existence and the uniform $C^{\infty}$-estimates follow from the standard nonlinear parabolic PDE theory with a strictly oblique boundary condition (cf. \cite{Dong,LSU,Ura}).

Assume that a capillary hypersurface $\S\subset \overline{\RR^{n+1}_+}$ is star-shaped with respect to the origin. 
One can reparametrize it as a graph over $\overline{\mathbb{S}^n_+}$. Namely, 
there exists a positive function $\rho$ defined on $\overline{\SS^n_+}$ such that
\begin{eqnarray*}
	\S=\left\{ \rho (X )X |    X\in \overline{\SS}^n_+\right\},
\end{eqnarray*}where $X:=(X_1,\ldots, X_n)$ is a local coordinate of $\overline{\SS^n_+}$.

We  denote  $\n^0$  the Levi-Civita connection on $\SS^n_+$ with respect to the standard round metric $\sigma:=g_{_{\SS^n_+}}$, $\p_i:=\p_{X_i}$,  $\sigma_{ij}:=\sigma(\p_{i},\p_{j})$,  $\rho_i:=\n^0_{i} \rho$, and $\rho_{ij}:=\n^0_{i}\n^0_j \rho$. The induced metric $g$ on $\S$ is given by
\begin{eqnarray*}
	g_{ij}=\rho^2\sigma_{ij}+ \rho_i\rho_j=e^{2\varphi}\left(\sigma_{ij}+\varphi_i\varphi_j\right),
\end{eqnarray*}where  $\varphi(X):=\log \rho(X)$. Its inverse $g^{-1}$ is given by
\begin{eqnarray*}
	g^{ij}=\frac{1}{\rho^2} \left(\sigma^{ij}-\frac{\rho^i\rho^j}{\rho^2+|\n^0 \rho|^2}\right)=e^{-2\varphi}\left( \sigma^{ij}-\frac{\varphi^i\varphi^j}{v^2}\right),
\end{eqnarray*}where $\rho^i:=\sigma^{ij}\rho_j$,  $\varphi^i:=\sigma^{ij}\varphi_j$ and $$v:=\sqrt{1+|\n^0 \varphi|^2}.$$
The unit outward normal vector field  on $\S$ is given by
\begin{eqnarray*}
	\nu=\frac{1}{v}\left( \p_\rho-\rho^{-2}\n^0  \rho \right)=\frac{1}{v}\left( \p_\rho-\rho^{-1}\n^0 \varphi\right).
\end{eqnarray*}
The second fundamental form $h:=(h_{ij})$ on $\S$ is
\begin{eqnarray*}
	h_{ij}=\frac{e^\varphi}{v }\left( \sigma_{ij}+\varphi_i\varphi_j-\varphi_{ij}\right),
\end{eqnarray*}and its Weingarten matrix $(h^i_j)$ is
\begin{eqnarray*}
	h_j^{i}=g^{ik}h_{kj}=\frac{1}{e^\varphi v }\left[ \delta^i_j-\left(\sigma^{ik}-\frac{\varphi^i\varphi^k}{v^2}\right)\varphi_{kj}\right].
\end{eqnarray*}
We also have for the mean curvature,
\begin{eqnarray}\label{express H-mean}
	H=\frac{1}{e^\varphi v}\left[ n-\left(\s^{ij}-\frac{\varphi^i\varphi^j}{v^2} \right)\varphi_{ij}\right].
\end{eqnarray}

Next, we derive the capillary boundary condition in \eqref{flow00} in terms of the radial function $\varphi$. We use the polar coordinate in 
the half-space, that is, for $x:=(x',x_{n+1})\in \RR^n\times [0,+\infty)$  and $X:=(\beta,\xi)\in [0,\frac{\pi}{2}]\times \SS^{n-1}$, we have
\begin{eqnarray*}
	x_{n+1}=r\cos\beta,\quad |x'|=r\sin\beta.
\end{eqnarray*} 
Then
\begin{eqnarray*}
	e_{n+1}=	\p_{x_{n+1}}=\cos\beta \p_r-\frac{\sin\beta}{r}\p_\beta.
\end{eqnarray*} 
In these coordinates the standard Euclidean metric is given by
$$ |dx|^2=dr^2+r^2\left(d\beta^2+\sin^2\beta g_{_{\SS^{n-1}}}\right).$$
It follows that
\begin{eqnarray}\label{express nu e}
	\langle \nu,e_{n+1}\rangle =\frac{1}{v}\left(\cos\beta+\sin\beta \n_{\p_\beta}^0   \varphi \right).
\end{eqnarray}

Along $\p\SS^n_+$, i.e., $\beta=\frac \pi 2$, it holds
\begin{eqnarray*}  	e=-e_{n+1}
	=\frac{1}{r}\p_\beta,
\end{eqnarray*}
then
\begin{eqnarray*}
	-\cos\theta =	\langle \nu,e\rangle =\left\langle \frac{1}{v}\left( \p_r-r^{-1}\n^0 \varphi\right), \frac{1}{r}\p_\beta \right\rangle=-\frac{\n^0_{\p_\beta}\varphi}{v} ,
\end{eqnarray*} that is,
\begin{eqnarray} \label{capillary bdry}
	\n^0_{\p_\beta} \varphi=\cos\theta  \sqrt{1+|\n^0\varphi|^2}.
\end{eqnarray}

In the following, we derive the full curvature estimates (which in turn are the $C^2$ estimates of $\varphi$) of  flow \eqref{flow00}.

To prove the existence of flow \eqref{flow00} for all time, we follow the approach in \cite[Section 4]{Marquardt2013} and \cite[Chapter V]{LSU}, which proves the  Schauder estimates without having direct access to a second derivative estimate of $\varphi$. 

Consider flow \eqref{flow00} for radial graphs over $\overline{\SS^n_+}$. It is known that the scalar function $\varphi(z,t):=\log \rho (X(z,t),t)$   satisfies
\begin{eqnarray}\label{scalar flow with capillary}
	\begin{array}{rcll}
		\p_t \varphi &=&G( \varphi_{ij} , \varphi_k),\quad &\text{ in } \SS^n_+\times [0,T^*),  \\
		\n_{\p_\beta}^0 \varphi&=&\cos\theta \sqrt{1+|\n^0\varphi|^2},\quad &\text{ on }
		\p\SS^n_+\times [0,T^*),\\
		\varphi(\cdot,0)&=&\varphi_0(\cdot),\quad &\text{ on } \SS^n_+,
	\end{array}
\end{eqnarray}
where $\p_\beta$ is the outward normal of $\p \SS^n_+\subset \overline{\SS^n_+}$,
 $\varphi_0 $ is the corresponding radial function of $x_0(M)$ over $\overline{\SS^n_+}$ and
\begin{eqnarray*}
	G(\varphi_{ij},\varphi_k):=    \frac{nv^2}{n-\left(\s^{ij}-\frac{\varphi^i\varphi^j}{v^2} \right)\varphi_{ij}} \left[1- \frac{\cos\theta}{v}\left( \cos\beta+\sin\beta \n^0_{\p_\beta}\varphi \right)  \right]-1. 
\end{eqnarray*}
Note that $G$ does not explicitly depend on the solution $\varphi$ itself, but only on its first and second derivatives, which is an important fact. Set
\begin{eqnarray*}
	G^{ij}:=\frac{\p G}{\p \varphi_{ij}}, ~~ G^k:=\frac{\p G}{\p \varphi_k}.
\end{eqnarray*}
Recall that  $T^*$ is the maximal time such that there exists some $\varphi \in C^{2,1}(\overline{\SS^n_+}\times [0,T^*))\cap C^\infty(\overline{\SS^n_+}\times (0,T^*))$ which solves \eqref{scalar flow with capillary}. In the sequel, we will prove a priori estimates for the solution $\varphi$ on $[0, T]$ for any $T<T^*$.
\subsection{Gradient estimates} First of all,  we give the estimate of $\dot{\varphi}:=\frac{\p\varphi}{\p t}$.
\begin{prop}\label{time derivative varphi est}
If $\varphi$ solves \eqref{scalar flow with capillary} and $\theta \in (0,\pi)$, then
	\begin{eqnarray*}
	&&	\min_{ \overline{\SS^n_+} } \frac{nv_0}{ e^{\varphi_0} H_0}\left[1- \frac{\cos\theta}{v}\left( \cos\beta+\sin\beta \n^0_{\p_\beta}\varphi_0 \right)  \right]\leq \dot{\varphi}+1\\&& \qquad \leq  \max_{\overline{\SS^n_+}}
	 \frac{nv_0}{ e^{\varphi_0} H_0}\left[1- \frac{\cos\theta}{v}\left( \cos\beta+\sin\beta \n^0_{\p_\beta}\varphi_0 \right)  \right], \quad \text{in}\quad \SS^n_+\times[0,T],
	\end{eqnarray*}
where $H_0:=H(\cdot,0)$ and $v_0:=v(\cdot,0)$.
\end{prop}

\begin{proof}
	Assume $\dot{\varphi}$ admits the maximum value at some point
	$(p_0,t_0)\in \overline{\SS^n_+}\times [0,T]$, that is
	\begin{eqnarray*}
		\dot{\varphi}(p_0,t_0)=\max_{ \overline{\SS^n_+}\times [0,T]} \dot{\varphi}.
	\end{eqnarray*}
We claim that $t_0=0$. If not, we have $t_0>0$. A
	direct computation gives 
	\begin{eqnarray*}
		\frac{\partial \dot{\varphi}}{\partial
			t}=\sum_{i,j=1}^{n}G^{ij}  \dot{\varphi}_{ij}+\sum_{k=1}^{n}G^k
	  \dot{\varphi}_k\quad \text{in}\quad\overline{\SS^n_+}  \times [0, T].
	\end{eqnarray*}
The maximum principle yields that  $p_0\in \partial
	\SS^n_+$. Choosing an orthonormal frame around $p_0$ such that
	$\{e_{\alpha}\}_{\alpha=1}^{n-1}$ are the tangential vector fields
	on $\partial \SS^n_+$ and $e_n:=\p_\beta$  the unit normal along $\p\SS^n_+\subset \overline{\SS^n_+} $. Then we have
	$(\dot{\varphi})_{\alpha}(p_0,t_0)=0$. Combining it with the boundary equation in \eqref{scalar flow with capillary} yields that
	\begin{eqnarray*}
		 (\dot{\varphi})_n(p_0,t_0)&=&\frac{d}{dt} \big(\cos\theta
			\sqrt{1+|\nabla^0 \varphi|^2}\big)
			\\&=&\sum_{i=1}^{n}\frac{ \varphi_i(\varphi_{i})_{t}}{
		\sqrt{1+|\nabla^0 \varphi|^2}	}\cos\theta=\frac{\varphi_n {(\dot{\varphi})}_{n}}{\sqrt{1+|\nabla^0 \varphi|^2}}\cos\theta=\cos^2\theta {(\dot{\varphi})}_{n},
	\end{eqnarray*}
Since  $|\cos\theta|<1$, we obtain $({\dot{\varphi}})_n(p_0,t_0) =0$, which is
	a contradiction to the Hopf boundary lemma. Therefore $t_0=0$, as a result,  
	\begin{eqnarray*}
 \dot{\varphi}\leq \max_{\overline{\SS^n_+}}\dot{\varphi_0}.
	\end{eqnarray*}Similarly, we have
\begin{eqnarray*}
 \dot{\varphi}\geq \min_{\overline{\SS^n_+}}\dot{\varphi_0}.
\end{eqnarray*}
Both, together with \begin{eqnarray*}
	G=\frac{nv}{e^\varphi H} \left[1- \frac{\cos\theta}{v}\left( \cos\beta+\sin\beta \n^0_{\p_\beta}\varphi \right)  \right]-1,
\end{eqnarray*}give the desired estimate. Hence we complete the proof.
\end{proof}

Proposition \ref{c1-est to flow} says that
the star-shapedness is preserved along flow \eqref{flow00}. In particular, combining with Proposition \ref{c0 est},  Proposition \ref{time derivative varphi est} yields the uniform gradient estimate of $\varphi$ in \eqref{scalar flow with capillary}. That is

\begin{prop}\label{grad est of varphi}
	If $\varphi$ solves \eqref{scalar flow with capillary}, then
	\begin{eqnarray*}
		\| \varphi \|_{C^1(\overline{\SS^n_+}\times [0,T])}\leq C,
	\end{eqnarray*}where $C>0$ depends only on the initial datum.
\end{prop}
\subsection{Schauder's estimate and long-time existence}

In this subsection, we follow the argument in \cite{Marquardt2013} (see also \cite{Xia}) to obtain the
higher regularity of flow \eqref{flow00}, i.e.  \eqref{scalar flow with capillary}. First, we estimate the H\"{o}lder norm of the spatial derivative of $\varphi$.
\begin{prop}\label{holder of spatial der}
If $\varphi$ solves \eqref{scalar flow with capillary}, then
there exists some $\gamma
>0 $ such that
\begin{eqnarray*}
	[\n^0\varphi]_{x,\gamma}+[\n^0 \varphi ]_{t,\frac{\gamma}{2}} \leq C,
\end{eqnarray*}where $[\psi]_{z,\gamma}$ denotes the $\gamma$-H\"{o}lder semi-norm of $\psi$ in $\overline{\SS^n_+}\times[0,T]$ with respect to the $z$-variable and  $C:=C(\|\varphi_0\|_{C^{2+\alpha}(\overline{\SS^n_+})}, n, \gamma   )$ .
\end{prop}
\begin{proof}
	The a priori estimates of $|\n^0 \varphi|$ and $|\p_t \varphi|$ imply the bound for $[\varphi]_{x,\gamma}$ and $[\varphi]_{t,\frac{\gamma}{2}}$, due to Proposition \ref{time derivative varphi est} and \ref{grad est of varphi}. Together with \cite[Chapter II, Lemma 3.1]{LSU}  one can give the bound for $[\n^0 \varphi ]_{t,\frac{\gamma}{2}} $ provided that we have a bound for $[\n^0\varphi]_{x,\gamma}$. Hence it suffices to bound $[\n^0\varphi]_{x,\gamma}$ in the following. In fact, we fix $t$ and rewrite \eqref{scalar flow with capillary} as a quasilinear elliptic PDE with a strictly oblique boundary value condition
	\begin{eqnarray}\label{ell pde of imcf}
	\text{div}_\sigma\left(\frac{\n^0 \varphi}{\sqrt{1+|\n^0 \varphi|^2}}\right)=\frac{n}{v} -\frac{n} {1+\dot{\varphi}}\left[ v -  {\cos\theta} \left( \cos\beta+\sin\beta \n^0_{\p_\beta}\varphi \right)  \right],
\end{eqnarray} and on $\p\SS^n_+$,	$$	\n_{\p_\beta}^0 \varphi=\cos\theta \sqrt{1+|\n^0\varphi|^2}.$$
 The right-hand side of \eqref{ell pde of imcf} is a bounded function in $x$,  due to Proposition \ref{time derivative varphi est} and \ref{grad est of varphi}. Thus for any $\Omega'\subset\subset \SS^n_+$, using \cite[Chapter 3, Theorem 14.1]{LU}, one can prove an interior estimate of form
 \begin{eqnarray*}
 	[\n^0 \varphi]_{\gamma,\Omega'}\leq C \left(\text{dist}_\s(\Omega',\p \SS^n_+), |\n^0 \varphi|, |\dot{\varphi}|\right),
 \end{eqnarray*}for some $\gamma>0$. And the boundary H\"{o}lder estimate of $\n^0\varphi$ can be obtained by adapting the similar procedure as in \cite[Chapter 10, Section 2]{LU}. Hence we prove the estimate for $[\n^0 \varphi]_{x,\gamma}$, which yields the conclusion.
	 
\end{proof}
Next, we estimate the H\"{o}lder norm of the time derivative of $\varphi$.
\begin{prop}\label{holder of time deri}
	If $\varphi$ solves \eqref{scalar flow with capillary}, then
	there exists some $\gamma
	>0 $ such that
	\begin{eqnarray*}
		[\p_t\varphi]_{x,\gamma}+[\p_t \varphi ]_{t,\frac{\gamma}{2}} \leq C,
	\end{eqnarray*}where $[\psi]_{z,\gamma}$ denotes the $\gamma$-H\"{o}lder semi-norm of $\psi$ in $\overline{\SS^n_+}\times[0,T]$ with respect to the $z$-variable and  $C:=C(\|\varphi_0\|_{C^{2+\alpha}(\overline{\SS^n_+})}, n, \gamma)$ .
\end{prop}
\begin{proof}Taking into account \eqref{express H-mean} and \eqref{express nu e}, the first equation in \eqref{scalar flow with capillary} is equivalent to 
 $$\dot{\varphi}= \frac{n(1+\cos\theta \<\nu,e\>)}{uH}-1, $$ where $$u:=\<x,\nu\>=\frac{e^\varphi}{v}.$$ 
		 Combining it with \eqref{evol of F}, \eqref{evo of support fun} and \eqref{evo of nu e}, and through some tedious calculations, we know that the evolution equation of $$ \Phi:=\dot \varphi$$ has a nice divergence structure with respect to the induced metric $g$, which is uniformly controlled (due to the uniform estimates for $\varphi$ and $\n^0\varphi$, see Proposition \ref{grad est of varphi}). Indeed,
		\begin{eqnarray}\label{eq of Phi}
  \p_t  \Phi &=&\text{div}_g \left(\frac{n (1+\cos\theta \<\nu,e\>)\n  \Phi}{H^2} \right) -\frac{2n(1+\cos\theta \<\nu,e\>)}{H^2(\Phi+1)}|\n  \Phi|^2 \\&&\notag
	+ \left\< \T+x-\frac{n\cos\theta}{H}e,\n  \Phi \right\>+ \frac{n}{H^2}\<\n (1+\cos\theta \<\nu,e\>) ,\n  \Phi\>. \end{eqnarray}
Moreover by using \eqref{neuman of bar u} and \eqref{neumann of F},  we know that $\Phi$ satisfies the Neumann boundary value condition, as
\begin{eqnarray*}
	\n_\mu \Phi=\n_\mu \left[\frac{n(1+\cos\theta \<\nu,e\>)}{uH}-1\right]=0.
\end{eqnarray*}In order to adopt the approach as in \cite[Chapter V]{LSU}, we use the weak formulation of the quasi-linear parabolic equation \eqref{eq of Phi}, that is
\begin{eqnarray}\notag
	\int_{t_0}^{t_1} \int_M \eta \p_t\Phi d\mu_tdt&=&	\int_{t_0}^{t_1} \int_M\Bigg[-\frac{n(1+\cos\theta \<\nu,e\>)}{H^2}\<\n \eta,\n \Phi\> \\&&\notag -  \frac{2n(1+\cos\theta \<\nu,e\>)}{H^2(\Phi+1)}|\n  \Phi|^2 \eta +\frac{n}{H^2} \<\n (1+\cos\theta \<\nu,e\>) ,\n  \Phi\>\eta\\&&+\eta \left\<\T+x-\frac{n\cos\theta}{H}e,\n\Phi \right\> \Bigg]d\mu_tdt,\label{weak form}
\end{eqnarray}for the test function  $\eta$ and $0<t_0<t_1<T$. 

Choose a specific test function $\eta$ defined by 
\begin{eqnarray}\label{eta}
	\eta :=\xi^2 (1+\cos\theta \<\nu,e\>) (\Phi+1),
\end{eqnarray}where $\xi\in C_c^\infty(B_\rho\times [0,T))$ is an arbitrary smooth cut-off function with values in $[0,1]$ for some small ball $B_\rho\subset M$. Then the left-hand side of \eqref{weak form} is equal to
\begin{eqnarray}
	\int_{t_0}^{t_1} \int_M \eta \p_t\Phi d\mu_tdt &=&\frac{1}{2}\int_M \xi^2(\Phi+1)^2 (1+\cos\theta \<\nu,e\>) d\mu_t\Big|_{t_0}^{t_1} \notag\\&&-	\int_{t_0}^{t_1} \int_M (\Phi+1)^2 \xi \frac{\p \xi}{\p t} (1+\cos\theta \<\nu,e\>)d\mu_tdt \notag \\&& {- \frac 1 2\int_{t_0}^{t_1} \int_M \xi^2(\Phi+1)^2  \frac{\p}{\p t} \big[  (1+\cos\theta \<\nu,e\>)  d\mu_t \big] dt}. ~~~ \label{term B}
\end{eqnarray}
The sum of the first three terms on the right-hand side in \eqref{weak form} equals
\begin{eqnarray*}
	\int_{t_0}^{t_1} \int_M \left[ -\frac{2n(1+\cos\theta \<\nu,e\>)^2}{H^2}(\Phi+1)\xi \<\n\xi,\n\Phi\>-\frac{3n(1+\cos\theta \<\nu,e\>)^2}{H^2}\xi^2 |\n \Phi|^2 \right]d\mu_tdt.
\end{eqnarray*}
Plugging these back into \eqref{weak form} yields that
\begin{eqnarray}
&&\frac{1}{2}\int_M \xi^2(\Phi+1)^2 (1+\cos\theta \<\nu,e\>) d\mu_t\Big|_{t_0}^{t_1} + 	\int_{t_0}^{t_1} \int_M\frac{3n(1+\cos\theta \<\nu,e\>)^2}{H^2}\xi^2 |\n \Phi|^2 d\mu_t dt \notag\\&=&	\int_{t_0}^{t_1} \int_M \left[ (\Phi+1)    \frac{\p \xi}{\p t}+\xi \left\< \T+x-\frac{n\cos\theta}{H}e,\n\Phi \right\> -\frac{2n(1+\cos\theta \<\nu,e\>)}{H^2}\<\n\xi,\n \Phi\> \right]\cdot \notag \\&& \qquad \xi(\Phi+1)(1+\cos\theta \<\nu,e\>)d\mu_tdt{+ \frac 1 2\int_{t_0}^{t_1} \int_M \xi^2(\Phi+1)^2  \frac{\p}{\p t} \big[  (1+\cos\theta \<\nu,e\>)  d\mu_t \big] dt.}~~~~~~~~~~\label{term B}
\end{eqnarray}
Notice  that Proposition \ref{upper bdd of F} implies that
\begin{eqnarray*}
	\frac{(1+\cos\theta\<\nu,e\>)^2\xi^2|\n \Phi|^2}{H^2}\geq 	\frac{(1+\cos\theta\<\nu,e\>)^2\xi^2|\n \Phi|^2}{\max H^2}.
\end{eqnarray*}For any $\varepsilon>0$, again using Proposition \ref{grad est of varphi}, Proposition \ref{lower bds of F}, Proposition \ref{upper bdd of F} and Young's inequality, we have 
\begin{eqnarray*}
&&\int_{t_0}^{t_1}\int_M \xi^2 \left\< \T+x-\frac{n\cos\theta}{H}e,\n\Phi \right\> (\Phi+1)(1+\cos\theta \<\nu,e\>)d\mu_tdt\\&\leq&  \int_{t_0}^{t_1}\int_M\left[\varepsilon \xi^2|\n\Phi|^2   (1+\cos\theta \<\nu,e\>)^2 +\frac{1}{4\varepsilon}  \xi^2\left| \T+x-\frac{n\cos\theta}{H}e\right|^2 (\Phi+1)^2\right]d\mu_tdt
\\&=&\int_{t_0}^{t_1}\int_M\left[ \varepsilon \xi^2|\n\Phi|^2 (1+\cos\theta \<\nu,e\>)^2+ C_{\varepsilon}\xi^2(\Phi+1)^2\right]d\mu_tdt
\end{eqnarray*}and
\begin{eqnarray*}
	&&-\int_{t_0}^{t_1}\int_M\frac{2n(1+\cos\theta \<\nu,e\>)}{H^2}\<\n\xi,\n \Phi\> \xi(\Phi+1)(1+\cos\theta \<\nu,e\>)d\mu_tdt\\&\leq&  \int_{t_0}^{t_1}\int_M\left[\varepsilon \xi^2|\n\Phi|^2   (1+\cos\theta \<\nu,e\>)^2 +\frac{n^2}{\varepsilon H^4}  |\n \xi|^2(\Phi+1)^2(1+\cos\theta \<\nu,e\>)^2\right]d\mu_tdt
	\\&=&\int_{t_0}^{t_1}\int_M\left[ \varepsilon \xi^2|\n\Phi|^2 (1+\cos\theta \<\nu,e\>)^2+ C_{\varepsilon}|\n\xi|^2(\Phi+1)^2\right]d\mu_tdt,
\end{eqnarray*}where $C_\varepsilon>0$ is a uniform constant.
Next, we handle the last term in  \eqref{term B}. From Proposition \ref{basic evolution eqs} (2) (3), we have
\begin{eqnarray}\label{time deri of area element}
&&  \frac{\p}{\p t} \big[(1+\cos\theta \<\nu,e\>) d\mu_t \big]   \\&=&  [\cos\theta \< -\n f+ h(e_i,\T)e_i,e\> +(1+\cos\theta \<\nu,e\>) (f H+\dive (\T) ] d\mu_t \notag
    \\&=& [  -\cos\theta \<\n f,e\> +(1+\cos\theta \<\nu,e\>) fH  + \dive((1+\cos\theta \<\nu,e\>)\T)  ]d\mu_t,\notag
\end{eqnarray}
where we have used the simple fact 
\begin{eqnarray*}
 \cos\theta  \< h(e_i,\T) e_i,e\>= \< \n(1+\cos\theta \<\nu,e\>), \T\>.
\end{eqnarray*} 
Noticing that  $$f=\frac{n(1+\cos\theta \<\nu,e\>)}{H}-\<x,\nu\>=u \Phi,$$ and 
\begin{eqnarray*}
  (1+\cos\theta \<\nu,e\>) f H= n  (1+\cos\theta \<\nu,e\>)^2 \frac{\Phi}{\Phi+1},
\end{eqnarray*} thus the last term in \eqref{term B} equals to
\begin{eqnarray}
 &&  - \frac{\cos\theta } 2 \int_{t_0}^{t_1} \int_M \xi^2(\Phi+1)^2  \n_{e^T} (u\Phi)d\mu_t dt +\frac n 2 \int_{t_0}^{t_1} \int_M (1+\cos\theta \<\nu,e\>)^2  \xi^2(\Phi+1)\Phi d\mu_t dt \notag \\&& +\frac 1 2 \int_{t_0}^{t_1} \int_M  \xi^2(\Phi+1)^2 \dive(  (1+\cos\theta \<\nu,e\>) \T)d\mu_t dt. \label{term B1}
\end{eqnarray}By integration by parts and Young's inequality,  for any $\varepsilon_1>0$, the third term in \eqref{term B1} satisfies 
\begin{eqnarray*}
&& \left|  \frac 1 2 \int_{t_0}^{t_1} \int_M  \xi^2(\Phi+1)^2 \dive(  (1+\cos\theta \<\nu,e\>) \T)d\mu_t dt \right|\\& = &  \left|  \int_{t_0}^{t_1} \int_M   (1+\cos\theta \<\nu,e\>)  \xi (\Phi+1)^2 \< \n \xi, \T\>d\mu_t dt \right. \\&& \left. 
+\int_{t_0}^{t_1} \int_M   (1+\cos\theta \<\nu,e\>)  \xi^2 (\Phi+1) \< \n \Phi, \T\>d\mu_t dt \right|
\\&\leq & \frac{\varepsilon_1}{2} \int_{t_0}^{t_1}\int_M \xi^2 |\n \Phi|^2 d\mu_t dt +C_{\varepsilon_1}  \int_{t_0}^{t_1}\int_M ( \xi^2+\xi|\n \xi|) (\Phi+1)^2 d\mu_t dt.
\end{eqnarray*}
Similarly, using again integration by part and Young's inequality,  the first term in \eqref{term B1} satisfies 
\begin{eqnarray*}
&& \left|  - \frac{\cos\theta } 2 \int_{t_0}^{t_1} \int_M \xi^2(\Phi+1)^2  \n_{e^T} (u\Phi)d\mu_t dt \right|  \\& =& \left| \int_{t_0}^{t_1}\int_M  {\cos\theta} \left[ \xi ( \n_{e^T} \xi ) u (\Phi+1)^2 \Phi +    \xi^2 u (\Phi+1)\Phi (\n_{e^T} \Phi)  \right] d\mu_t dt \right|
 \\&\leq & \frac{\varepsilon_1}{2} \int_{t_0}^{t_1}\int_M \xi^2 |\n \Phi|^2 d\mu_t dt  +C_{\varepsilon_1}  \int_{t_0}^{t_1}\int_M ( \xi^2+\xi|\n \xi|) (\Phi+1)^2 d\mu_t dt,
\end{eqnarray*}where we have used the fact that $u$ and $\Phi$ are uniformly bounded in the last inequality.  Now by substituting the above two terms back 
into \eqref{term B1}, we conclude that
\begin{eqnarray}\label{term B2}
    |\eqref{term B1}|\leq {\varepsilon_1}\int_{t_0}^{t_1}\int_M \xi^2 |\n \Phi|^2 d\mu_t dt +C  \int_{t_0}^{t_1}\int_M ( \xi^2+|\n \xi|^2) (\Phi+1)^2 d\mu_t dt,
\end{eqnarray}   which gives an estimate for the last term in \eqref{term B}. 
By choosing $\varepsilon, \varepsilon_1>0$ small enough, say $\varepsilon:=\frac{n}{\max H^2}$ and $\varepsilon_1:=\frac{n(1-\cos\theta)^2}{2\max H^2} $ and noticing $2\geq 1+\cos\theta \<\nu,e\>\geq 1-\cos\theta >0$, we conclude that
\begin{eqnarray*}
	&&\frac{1}{2}(1-\cos\theta)\int_M \xi^2(\Phi+1)^2   d\mu_t\Big|_{t_0}^{t_1} +\frac{n(1-\cos\theta  )^2}{2\max H^2} 	\int_{t_0}^{t_1} \int_M\xi^2 |\n \Phi|^2 d\mu_t dt\\&\leq &	C \int_{t_0}^{t_1} \int_M   (\Phi+1) ^2\left[|\n \xi|^2 +\xi^2+\xi \left|\frac{\p\xi}{\p t}\right|(1+\cos\theta \<\nu,e\>)\right] d\mu_tdt.
\end{eqnarray*}
This inequality has the same type of the one in \cite[Chapter V, \textsection 1, Eq.(1.13)]{LSU}, therefore, similar to \cite[Chapter V, \textsection 1]{LSU} (interior estimate) and \cite[Chapter V, \textsection 7, Page 478]{LSU} (boundary estimate) or \cite[Chapter XIII, section 6]{Lie}, the bound for
	$[\Phi]_{x,\gamma}=[\dot \varphi]_{x,\gamma}$ and $[\Phi]_{t,\frac \gamma 2}=[\dot\varphi]_{t,\frac \gamma 2}$  can be obtained. Moreover, all local interior
	and boundary estimates are independent of  $T$. The global estimates follow from the local results and a covering argument, which also does not depend on $T$.	
	\end{proof}
The H\"{o}lder estimate of gradient $\varphi$ implies the H\"{o}lder estimate for the mean curvature.
\begin{prop}\label{holder of mean cur}
	If $\varphi$ solves \eqref{scalar flow with capillary}, then
	there exists some $\gamma
	>0 $ such that
	\begin{eqnarray}\label{holder of mean curv}
		[H]_{x,\gamma}+[H ]_{t,\frac{\gamma}{2}} \leq C,
	\end{eqnarray}where $[\psi]_{z,\gamma}$ denotes the $\gamma$-H\"{o}lder semi-norm of $\psi$ in $\overline{\SS^n_+} \times[0,T]$ with respect to the $z$-variable and  $C:=C(\|\varphi_0\|_{C^{2+\alpha}(\overline{\SS^n_+})}, n, \gamma)$.
\end{prop}
\begin{proof}
	 From \eqref{scalar flow with capillary}, we know
	 \begin{eqnarray*}
	 	e^\varphi H= \frac{n}{1+\dot\varphi} \left[\sqrt{1+|\n^0\varphi|^2}-  {\cos\theta} \left( \cos\beta+\sin\beta \n^0_{\p_\beta}\varphi \right)  \right].
	 \end{eqnarray*}
	 Hence the conclusion \eqref{holder of mean curv} follows from the H\"{o}lder estimates for $|\n^0 \varphi|, \dot \varphi$ and $\varphi$ in Proposition \ref{holder of spatial der} and Proposition \ref{holder of time deri}, while the H\"{o}lder estimate for $\varphi$ follows trivially from Proposition \ref{time derivative varphi est} and Proposition \ref{grad est of varphi}.
\end{proof}
With the preparation above, we obtain the full second-order and high-order derivative estimates for $\varphi$.
\begin{prop}
	If $\varphi$ solves \eqref{scalar flow with capillary}, then for every $t_0\in(0,T)$,
	there exists some $\gamma
	>0 $ such that
	\begin{eqnarray}\label{2nd order est}
	\|\varphi\|_{C^{2+\gamma,1+\frac{\gamma}{2}}(\overline{\SS^n_+}\times[0,T])}\leq C(\|\varphi_0\|_{C^{2+\alpha}(\overline{\SS^n_+})}, n, \gamma   ),
	\end{eqnarray}and
	\begin{eqnarray}\label{kth order est}
	\|\varphi\|_{C^{2l+\gamma,l+\frac{\gamma}{2}}(\overline{\SS^n_+}\times[t_0,T])}\leq C(\|\varphi(\cdot,t_0)\|_{C^{2l+\alpha}(\overline{\SS^n_+})}, n, \gamma   ),
\end{eqnarray}for any $2\leq l\in \NN$.
\end{prop}
\begin{proof}
To get the Schauder estimates \eqref{2nd order est}, we rewrite the evolution equation for $\varphi$ with respect to the induced metric $g$, which is uniformly controlled (due to the uniform estimate for $\varphi$ and $\n^0\varphi$, see Proposition \ref{grad est of varphi}). Noticing that $v:=\sqrt{1+|\n^0\varphi|^2}$ and
\begin{eqnarray*}
e^\varphi v	H=  n-\left(\s^{ij}-\frac{\varphi^i\varphi^j}{v^2} \right)\varphi_{ij}=n-e^{2\varphi} \Delta_g \varphi,
\end{eqnarray*} we have
\begin{eqnarray*}
	\p_t\varphi &=& \frac{nv}{e^\varphi H}\left[1- \frac{\cos\theta}{v}\left( \cos\beta+\sin\beta \n^0_{\p_\beta}\varphi \right)  \right]-1
	\\&=&-\frac{n}{e^{2\varphi}H^2}(n-e^{2\varphi}\Delta_g\varphi) \left[1- \frac{\cos\theta}{v}\left( \cos\beta+\sin\beta \n^0_{\p_\beta}\varphi \right)  \right]\\&&+\frac{2nv}{e^\varphi H}\left[1- \frac{\cos\theta}{v}\left( \cos\beta+\sin\beta \n^0_{\p_\beta}\varphi \right)  \right]-1.
\end{eqnarray*} It follows 
\begin{eqnarray}
	\p_t\varphi&=&\label{eq of varphi}
	\frac{n}{H^2}\left[1- \frac{\cos\theta}{v}\left( \cos\beta+\sin\beta \n^0_{\p_\beta}\varphi \right)  \right] \Delta_g \varphi \\&& \notag +\left(\frac{2nv}{e^\varphi H}-\frac{n^2}{e^{2\varphi}H^2}\right)\left[1- \frac{\cos\theta}{v}\left( \cos\beta+\sin\beta \n^0_{\p_\beta}\varphi \right)  \right]-1,
\end{eqnarray} 
	which is a linear, uniformly parabolic equation with H\"{o}lder coefficients, due to Proposition \ref{holder of mean cur} and Proposition \ref{holder of spatial der}. Moreover, the Neumann boundary value condition in \eqref{scalar flow with capillary} satisfies the strictly oblique property (due to $|\cos\theta|<1$) and its coefficient has uniform H\"{o}lder bounds (due to again Proposition \ref{holder of spatial der}). Hence \eqref{2nd order est} follows from the linear parabolic PDE theory with a strictly oblique boundary value problem (cf. \cite[Chapter V, Theorem 4.9 and Theorem 4.23]{Lie}, see also \cite[Chapter IV, Theorem 5.3]{LSU}).

Next, we show \eqref{kth order est}. By differentiating both sides of \eqref{eq of varphi} with respect to $t$ and
	$\p_{j}$, $1\leq j\leq n$, respectively, one can easily
	get evolution equations of $\dot \varphi$ and $\n^0_j\varphi $ respectively, which,
	using the estimate \eqref{2nd order est}, can be treated as uniformly
	parabolic PDEs on the time interval $[t_0,T]$. At the initial time
	$t_0$, all compatibility conditions are satisfied and the initial
	function $\varphi(\cdot, t_0)$ is smooth, which implies the
	$C^{3+\gamma,(3+\gamma)/2}$-estimate for $\n^0_j\varphi$ and the
	$C^{2+\gamma,(2+\gamma)/2}$-estimate for $\dot \varphi$. So, we have the
	$C^{4+\gamma,(4+\gamma)/2}$-estimate for $\varphi$. From \cite[Chapter 4,
	Theorem 4.3, Exercise 4.5]{Lie} and the above argument, it is not
	difficult to know that the constants are independent of the time $T$. Higher
	regularity of \eqref{kth order est} can be proven by induction over $2\leq l\in \NN$. Hence we complete the proof.
\end{proof}

With the estimates given above, the long-time existence of inverse mean curvature type flow \eqref{scalar flow with capillary} or \eqref{flow00}   follows directly. That is.

\begin{cor} \label{longtime ex imcf}
If $\S_0$ is star-shaped and strictly mean convex capillary hypersurface with a contact angle $\theta \in (0, \pi)$, then the solution of flow \eqref{flow00}  exists for all time with a uniform $C^\infty$-estimates. 
\end{cor}

\
\

We are now ready to prove Theorem \ref{k-convex conv} by using the above uniform a priori estimates.

\begin{proof}[\textit{Proof of Theorem \ref{k-convex conv}}]
Following the parabolic theory with a strictly oblique boundary value condition (cf. \cite{Dong,Lie,Ura}), we know that the solution of equation \eqref{flow00} exists for all time with a uniform $C^\infty$-estimate. And the convergence result of Theorem \ref{k-convex conv} can be proved in a similar way as \cite[Proposition 4.12]{WWX2022}, by adopting the monotonicity property of quermassintegrals along flow \eqref{flow00}.
\end{proof}

\
Finally, by applying the convergence result of flow \eqref{flow00} and the monotonicity of $\V_{2,\theta}$, i.e., Theorem \ref{k-convex conv} and \cite[Proposition 3.1]{WWX2022}, one can complete the proof of Theorem \ref{thm 2} for the strictly mean convex case. When $\S$ is only mean convex, we use the approximation argument, as described in \cite{GL09} to prove Theorem \ref{thm 2}. The equality characterization in Theorem \ref{thm 2} can be established using the same approach as in \cite{SWX, WWX2022}. For the sake of conciseness, we will omit it here and leave it to the interested reader.
  
    \
    
  \noindent\textit{Acknowledgment:} 
  The authors would like to thank the referee for his/her careful reading and pointing out a missing term in the previous version.
  
L.W. was partially supported by: MIUR excellence project: Department of Mathematics, University of Rome "Tor Vergata" CUP E83C18000100006 and NSFC (Grant No. 12201003). Part of this work was done while L.W.  was visiting the Department of Mathematics at the University of Rome "Tor Vergata".  He would like to express his sincere gratitude to Prof. ssa G. Tarantello for her constant encouragement and support, and to the department for its hospitality. C.X. was partially supported by NSFC (Grant No. 12271449, 12126102).


\begin{thebibliography}{10}  
\bibitem{AFM}  Virginia Agostiniani, Mattia Fogagnolo, Lorenzo Mazzieri, Minkowski inequalities via nonlinear potential theory.	\textit{Arch. Ration. Mech. Anal.} \textbf{244} (2022), no. 1, 51--85.
	
 \bibitem{BGL}Simon Brendle, Pengfei Guan, Junfang Li, An inverse curvature type hypersurface flow in space forms, (Private note).

 \bibitem{BHW}Simon Brendle, Pei-Ken Hung, Mu-Tao Wang, A Minkowski-type inequality for hypersurfaces in the anti–de Sitter–Schwarzschild manifold. \textit{Comm. Pure Appl. Math.} \textbf{69} (2016), no. 1, 124–144.

 \bibitem{CGLS}Chuanqiang Chen, Pengfei Guan, Junfang Li, Julian Scheuer, A fully nonlinear flow and quermassintegral inequalities. \textit{Pure Appl. Math Q.} \textbf{18}, no. 2, p. 437--461, (2022).


\bibitem{Dong}Guangchang Dong, Initial and nonlinear oblique boundary value problems for fully nonlinear parabolic equations.	\textit{J. Partial Differential Equations Ser.} A \textbf{1} (1988), no. 2, 12--42.

  \bibitem{Finn} Robert Finn, Equilibrium Capillary Surfaces. Springer, New York (1986).


  \bibitem{Ger} Claus Gerhardt, Flow of nonconvex hypersurfaces into spheres. \textit{J. Differential Geom.} \textbf{32} (1990), no. 1, 299--314.
  
\bibitem{Glau}Federico Glaudo, Minkowski-type inequality for nearly spherical domains. \textit{Adv. Math.} \textbf{408} (2022), Paper No. 108595, 33 pp.

	
\bibitem{GL09}Pengfei  Guan, Junfang Li, The quermassintegral inequalities for $k$-convex star-shaped domains. \textit{Adv. Math.} \textbf{221} (2009), no. 5, 1725--1732.

\bibitem{GL15} Pengfei  Guan,  Junfang Li, A mean curvature type flow in space forms. \textit{Int. Math. Res. Not.} 2015, no. {13}, 4716--4740.
 
\bibitem{GLW} Pengfei Guan,  Junfang Li, Mu-Tao Wang, A volume-preserving flow and the isoperimetric problem in warped product spaces. \textit{Trans. Amer. Math. Soc.} \textbf{372} (2019), no. 4, 2777--2798.

\bibitem{Wei_el}Yingxiang Hu, Yong Wei, Bo Yang, Tailong Zhou, A complete family of Alexandrov-Fenchel inequalities for convex capillary hypersurfaces in the half-space.	\textit{Math. Ann.} (2024). https://doi.org/10.1007/s00208-024-02841-9



\bibitem{Hui}Gerhard Huisken,  An Isoperimetric Concept for the Mass in General Relativity (2009). https://www.ias.edu/video/marston-morse-isoperimetric-concept-mass-general-relativity



\bibitem{LSU} Olga A. Ladyzhenskaya, V. A. Solonnikov, Nina N. Ural'ceva, \textit{Linear and quasilinear equations of parabolic type.} Translations of Mathematical Monographs, Vol. \textbf{23} American Mathematical Society, Providence, R.I. 1968 xi+648 pp.
	
\bibitem{LU} Olga A. Ladyzhenskaya, Nina N. Uraltseva, \textit{Linear and quasilinear elliptic equations}. Translated from the Russian by Scripta Technica, Inc. Translation editor: Leon Ehrenpreis Academic Press, New York-London 1968 xviii+495.

 

\bibitem{Lie} Gary M. Lieberman, \textit{Second order parabolic differential equations}. World Scientific Publishing Co., Inc., River Edge, NJ, 1996. 
 
\bibitem{Marquardt2013} Thomas Marquardt, Inverse mean curvature flow for star-shaped hypersurfaces evolving in a cone. \textit{J. Geom. Anal.} \textbf{23} (2013), no. 3, 1303--1313.

\bibitem{MWW} Xinqun Mei, Guofang Wang, Liangjun Weng, A constrained mean curvature flow and Alexandrov-Fenchel inequalities. \textit{Int. Math. Res. Not. IMRN} \textbf{2024}, no. 1, 152--174. 


\bibitem{MWWX} Xinqun Mei, Guofang Wang, Liangjun Weng,  Chao Xia, Alexandrov-Fenchel inequalities for convex hypersurfaces in the half-space with capillary boundary  II. Preprint.

\bibitem{Sch2}Julian Scheuer, Minkowski inequalities and constrained inverse curvature flows in warped spaces. \textit{Adv. Calc. Var.} \textbf{15} (2022), no. 4, 735--748. 

\bibitem{SWX}Julian Scheuer, Guofang Wang, Chao  Xia, Alexandrov-Fenchel inequalities for convex hypersurfaces with free boundary in a ball. \textit{J. Differential Geom.} \textbf{120} (2022),  no. 2, 345--373.	
 
 \bibitem{Schneider} Rolf Schneider, Convex bodies: the Brunn-Minkowski theory. Second expanded edition. Encyclopedia of Mathematics and its Applications, 151. Cambridge University Press, Cambridge, 2014.
 
\bibitem{Ura}Nina N. Ural'tseva,  A nonlinear problem with an oblique derivative for parabolic equations.   \textit{Zap. Nauchn. Sem. Leningrad. Otdel. Mat. Inst. Steklov.} (LOMI) \textbf{188} (1991).

\bibitem{Urb90}John I. E.  Urbas, On the expansion of star-shaped hypersurfaces by symmetric functions of their principal curvatures. \textit{Math. Z.} \textbf{205} (1990), no. 3, 355--372.
 
\bibitem{WWX2022}Guofang Wang, Liangjun Weng, Chao Xia, Alexandrov-Fenchel inequalities for convex hypersurfaces in the half-space with capillary boundary. \textit{Math. Ann.} \textbf{388}, 2121--2154 (2024).

\bibitem{Wei}Yong Wei, On the Minkowski-type inequality for outward minimizing hypersurfaces in Schwarzschild space. \textit{Calc. Var. Partial Differential Equations.} \textbf{57} (2018), no. 2, Paper No. 46, 17 pp.


\bibitem{WeX21}Liangjun Weng, Chao Xia, The Alexandrov-Fenchel inequalities for convex hypersurfaces with capillary boundary in a ball. \textit{Trans. Amer. Math. Soc.} \textbf{375} (2022), no. 12, 8851--8883.

\bibitem{Xia} Chao Xia, Inverse anisotropic mean curvature flow and a Minkowski-type inequality. \textit{Adv. Math.} \textbf{315} (2017), 102--129. 

\end{thebibliography}
\end{document}